\documentclass[10pt,a4paper]{article}
\usepackage{latexsym,amsthm,amssymb,amscd,amsmath}
\newcounter{alphthm}
\theoremstyle{plain}
\newtheorem{thm}{Theorem}[section]

 \newtheorem{exam}[thm]{Example}
 \newtheorem{prop}[thm]{Proposition}
 \theoremstyle{definition}
 \newtheorem{defn}[thm]{Definition}
 \theoremstyle{remark}
 \newtheorem{rem}[thm]{Remark}
 \numberwithin{equation}{subsection}

\begin{document}

\title {Two Fixed-Point Theorems For Special Mappings
 }

\author{ A. Beiranvand, S. Moradi\footnote{First author}, M. Omid and H. Pazandeh\\\\
Faculty  of Science, Department of Mathematics\\
Arak University, Arak,  Iran\\
\date{}
 } \maketitle
\begin{abstract}
In this paper, we study the existence of fixed points for
 mappings defined on complete (compact) metric space ($X,d$)
 satisfying a general contractive (contraction) inequality depended
 on another function. These conditions are analogous to Banach
 conditions.

\end{abstract}


\section{Introduction}
The first important result on fixed points for contractive-type
 mapping was the well-known Banach's Contraction Principle
 appeared in explicit form in Banach's thesis in 1922, where it was
 used to establish the existence of a solution for an integral
 equation. This paper published for the first time in 1922 in [1].
 In the general setting of complete metric spaces, this theorem
 runs as follows (see [3, Theorem 2.1] or [8, Theorem 1.2.2]).

\begin{thm}$($\textbf{Banach's Contraction Principle}$)$
 Let $(X,d)$ be a complete metric space and $S:X \longrightarrow
 X$ be a contraction $($there exists $k \in ]0,1[$ such that for
 each $x,y \in X$; $d(Sx,Sy) \leq kd(x,y)$$)$. Then $S$ has a unique
 fixed point in $X$, and for each $x_{0} \in X$ the sequence of
 iterates $\{S^{n}x_{0}\}$ converges to this fixed point.
\end{thm}

After this classical result Kannan in [2] analyzed a substantially
new type of contractive condition. Since then there have been
many theorems dealing with mappings satisfying various types of
contractive inequalities. Such conditions involve linear and
nonlinear expressions (rational, irrational, and of general
type). The intrested reader who wants to know more about this
matter is recommended to go deep into the survey articles by
Rhoades [5,6,7] and Meszaros [4], and into the references therein.

Another result on fixed points for contractive-type mapping is
generally attributed to Edelstein (1962) who actually obtained
slightly more general versions.

In the general setting of compact metric spaces this result runs
as followes (see [3, Theorem 2.2]).

\begin{thm}
 Let $(X,d)$ be a compact metric space and $S:X \longrightarrow
 X$ be a contractive $($for every $x,y \in X$ such that $x \neq y$; $d(Sx,Sy) < d(x,y)$$)$. Then $S$ has a unique
 fixed point in $X$, and for any $x_{0} \in X$ the sequence of
 iterates $\{S^{n}x_{0}\}$ converges to this fixed point.
\end{thm}

The aim of this paper is to analyze the existence of fixed points
for mapping $S$ defined on a complete (compact) metric space
$(X,d)$ such that is $T-contraction$ ($T-contractive$). See
Theorem 2.6 and Theorem 2.9 below.

First we introduce the $T-contraction$ and $T-contractive$
functions and then we extend the Banach-Contraction Principle and
Theorem 1.2.

At the end of paper some properties and examples concerning this
kind of contractions and contractives are given.

In the sequel, $\Bbb{N}$ will represent the set of natural
numbers.


\section{Definitions and Main Results}

The following theorems (Theorem 2.6 and Theorem 2.9) are the main
results of this paper. In the first, we define some new
definitions.

\begin{defn}
 Let $(X,d)$ be a metric space and $T,S:X \longrightarrow X$ be
 two functions. A mapping $S$ is said to be a $T-contraction$ if
 there exists $k \in ]0,1[$ such that for
 \[d(TSx,TSy) \leq kd(Tx,Ty) \quad\quad \forall x,y \in X.\]
\end{defn}
\textbf{Note 1}. By taking $Tx=x$ (T is identity function)
$T-contraction$ and contraction are equivalent.

The following example shows that $T-contraction$ functions maybe
not contraction.

\begin{exam}
Let $X=[1,+\infty)$ with metric induced by $\Bbb{R}$:
$d(x,y)=|x-y|$. We consider two mappings $T,S:X \longrightarrow
X$ by $Tx=\frac{1}{x}+1$ and $Sx=2x$. Obviously $S$ is not
contraction but $S$ is $T-contraction$, because:
\[ \big{|}TSx-TSy\big{|}=\big{|}\frac{1}{2x}+1-\frac{1}{2y}-1\big{|}=\big{|}\frac{1}{2x}-\frac{1}{2y}\big{|}
\leq \frac{1}{2}\big{|}\frac{1}{x}-\frac{1}{y}\big{|}
=\frac{1}{2}\big{|}\frac{1}{x}+1-\frac{1}{y}-1\big{|}=\frac{1}{2}\big{|}Tx-Ty\big{|}.\]
\end{exam}

\begin{defn}
 Let $(X,d)$ be a metric space. A mapping $T:X \longrightarrow X$
 is said sequentially convergent if we have, for every sequence
 $\{y_{n}\}$, if $\{Ty_{n}\}$ is convergence then $\{y_{n}\}$ also is
 convergence.
\end{defn}

\begin{defn}
 Let $(X,d)$ be a metric space. A mapping $T:X \longrightarrow X$
 is said subsequentially convergent if we have, for every sequence
 $\{y_{n}\}$, if $\{Ty_{n}\}$ is convergence then $\{y_{n}\}$ has a convergent subsequence.
\end{defn}

\begin{prop}
If $(X,d)$ be a compact metric space, then every function $T:X
\longrightarrow X$ is subsequentially convergent and every
continuous function $T:X \longrightarrow X$ is sequentially
convergent.
\end{prop}

\begin{thm}
 Let $(X,d)$ be a complete metric space and $T:X \longrightarrow
 X$ be a one-to-one, continuous and subsequentially convergent mapping. Then for every $T-contraction$
 continuous function $S:X \longrightarrow X$, $S$ has a unique fixed point. Also if
 $T$ is a sequentially convergent, then for each $x_{0}\in
 X$, the sequence of iterates $\{S^{n}x_{0}\}$ converges to this
 fixed point.
\end{thm}
\begin{proof}
For every $x_{1}$ and $x_{2}$ in $X$,
\begin{align*}
 d(Tx_{1},Tx_{2}) & \leq d(Tx_{1},TSx_{1})+d(TSx_{1},TSx_{2})+d(TSx_{2},Tx_{2}) \\&
 \leq d(Tx_{1},TSx_{1})+k d(Tx_{1},Tx_{2})
+d(TSx_{2},Tx_{2}),
\end{align*}
 so
\begin{equation}
d(Tx_{1},Tx_{2}) \leq
\frac{1}{1-k}[d(Tx_{1},TSx_{1})+d(TSx_{2},Tx_{2})]
\end{equation}
Now select $x_{0} \in X$ and define the iterative sequence
$\{x_{n}\}$ by $x_{n+1}=Sx_{n}$ (equivalently,
$x_{n}=S^{n}x_{0}$), $n=1,2,3,...$. By (2.0.1) for any indices
$m,n \in \Bbb{N}$,
\begin{align*}
& d(Tx_{n},Tx_{m})= d(TS^{n}x_{0},TS^{m}x_{0}) \\
& \leq
\frac{1}{1-k}[d(TS^{n}x_{0},TS^{n+1}x_{0})+d(TS^{m+1}x_{0},TS^{m}x_{0})]
\\ & \leq \frac{1}{1-k}[k^{n}d(Tx_{0},TSx_{0})+k^{m}d(TSx_{0},Tx_{0})]
\end{align*}
hence
\begin{equation}
d(TS^{n}x_{0},TS^{m}x_{0}) \leq
\frac{k^{n}+k^{m}}{1-k}d(Tx_{0},TSx_{0}).
\end{equation}
Relation (2.0.2) and condition $0 < k < 1$ show that
$\{TS^{n}x_{0}\}$ is a Cauchy sequence, and since $X$ is complete
there exists $a \in X$ such that
\begin{equation}\underset{n
\rightarrow \infty}\lim TS^{n}x_{0}=a.
\end{equation}
Since $T$ is subsequentially convergent $\{S^{n}x_{0}\}$ has a
convergent subsequence. So, there exist $b \in X$ and
$\{n_{k}\}_{k=1}^{\infty}$ such that $\underset{k \rightarrow
\infty}\lim S^{n_{k}}x_{0}=b$. Hence, $\underset{k \rightarrow
\infty}\lim TS^{n_{k}}x_{0}=Tb$, and by (2.0.3), we conclude that
\begin{equation}
Tb=a.
\end{equation}
Since $S$ is continuous and $\underset{k \rightarrow \infty}\lim
S^{n_{k}}x_{0}=b$, then $\underset{k \rightarrow \infty}\lim
S^{n_{k}+1}x_{0}=Sb$ and so
\[\underset{k \rightarrow \infty}\lim
TS^{n_{k}+1}x_{0}=TSb.\] Again by (2.0.3), $\underset{k
\rightarrow \infty}\lim TS^{n_{k}+1}x_{0}=a$ and therefore
$TSb=a$. Since $T$
is one-to-one and by (2.0.4), Sb=b. So, $S$ has a fixed point.\\
Since $T$ is one-to-one and $S$ is $T-contraction$, $S$ has a
unique fixed point.
\end{proof}

\begin{rem}
By above theorem and taking $Tx=x$ (T is identity function), we
can conclude Theorem 1.1.
\end{rem}

\begin{defn}
 Let $(X,d)$ be a metric space and $T,S:X \longrightarrow X$ be
 two functions. A mapping $S$ is said to be a $T-contractive$ if
 for every $x,y \in X$ such that $Tx \neq Ty$ then $d(TSx,TSy) < d(Tx,Ty)$.
\end{defn}

Obviously, every $T-contraction$ function is $T-contractive$ but
the converse is not true. For example if $X=[1,+\infty)$,
$d(x,y)=|x-y|$, $Sx=\sqrt{x}$ and $Tx=x$ then $S$ is
$T-contractive$ but $S$ is not $T-contraction$.

\begin{thm}
 Let $(X,d)$ be a compact metric space and $T:X \longrightarrow
 X$ be a one-to-one and continuous mapping. Then for every $T-contractive$ function
 $S:X \longrightarrow X$, $S$ has a unique fixed point. Also for any $x_{0} \in X$ the sequence of
 iterates $\{S^{n}x_{0}\}$ converges to this fixed point.
\end{thm}
\begin{proof}
\textbf{Step 1.} In the first we show that $S$ is continuous.

 Let $\underset{n \rightarrow \infty}\lim x_{n}=x$. We prove that $\underset{n \rightarrow \infty}\lim
 Sx_{n}=Sx$. Since $S$ is $T-contractive$ $d(TSx_{n},TSx) \leq
 d(Tx_{n},Tx)$ and this shows that $\underset{n \rightarrow \infty}\lim TSx_{n}=TSx$ (because $T$ is
 continuous).\\
 Let $\{Sx_{n_k}\}$ be an arbitary convergence subsequence of
 $\{Sx_{n}\}$. There exists a $y \in X$ such that $\underset{k \rightarrow \infty}\lim
 Sx_{n_k}=y$. Since $T$ is continuous so, $\underset{k \rightarrow \infty}\lim
 TSx_{n_k}=Ty$. By $\underset{n \rightarrow \infty}\lim
 TSx_{n}=TSx$, we conclude that $TSx=Ty$. Since $T$ is one-to-one
 so, $Sx=y$. Hence, every convergence subsequence of $\{Sx_{n}\}$
 converge to $Sx$. Since $X$ is a compact metric space $S$ is
 continuous.

 \textbf{Step 2.} Since $T$ and $S$ are continuous, the function
 $\varphi: X \longrightarrow [0,+\infty)$ defined by
 $\varphi(y)=d(TSy,Ty)$ is continuous on $X$ and hence by
 compactness attains its minimum, say at $x \in X$. If $Sx \neq x$
 then
 \[\varphi(Sx)=d(TS^{2}x,TSx) < d(TSx,Tx)\]
 is a contradiction. So $Sx=x$.\\
 Now let $x_{0} \in X$ and set $a_{n}=d(TS^{n}x_{0},Tx)$. Since
 \[a_{n+1}=d(TS^{n+1}x_{0},Tx)=d(TS^{n+1}x_{0},TSx)
 \leq d(TS^{n}x_{0},Tx)=a_{n},\]
 then $\{a_n \}$ is a nonincreasing sequence of nonnegative real numbers and so has a limit, say
 $a$.\\
 By compactness, $\{TS^{n}x_{0}\}$ has a convergent subsequence $\{TS^{n_k}x_{0}\}$; say
 \begin{equation}
 \lim TS^{n_k}x_{0}=z.
 \end{equation}
 Since $T$ is sequentially convergence (by Note 2) for a $w \in X$ we
 have
 \begin{equation}
 \lim S^{n_k}x_{0}=w.
 \end{equation}
 By (2.0.5) and (2.0.6), $Tw=z$. So $d(Tw,Tx)=a$. Now we show that
 $Sw=x$. If $Sw \neq x$, then
 \begin{align*}
 a=\lim d(TS^{n}x_{0},Tx) & =\lim d(TS^{n_k}x_{0},Tx)=d(TSw,Tx) \\
 & =d(TSw,TSx) < d(Tw,Tx)=a
 \end{align*}
 that is contradiction. So $Sw=x$ and hence,
 \[a=\lim d(TS^{n_k+1}x_{0},Tx)=d(TSw,Tx)=0. \]
 Therefore, $\lim TS^{n}x_{0}=Tx_{0}$. Since $T$ is
 sequentially convergence (by Proposition 2), then $\lim S^{n}x_{0}=x$.
\end{proof}

Similar to Remark 2.7, we can conclude Theorem 1.2.

\begin{rem}
 In Theorem 2.6 (Theorem 2.9) if $S^{n}$ is
 $T-contraction$$(T-contractive)$, then $S^{n}$ has a unique fixed
 point and we conclude that $S$ has a unique fixed point. So, we
 can replace $S$ by $S^{n}$ in Theorem 2.6 (Theorem 2.9).
\end{rem}

We know that for some function $S$, $S$ is not
$T-contraction$$(T-contractive)$, but for some $n \in \Bbb{N}$
$S^{n}$ is $T-contraction$ $(T-contractive)$ (see the following
example).


\section{Examples and Applications}

In this section we have some example about Theorem 2.6 and Theorem
2.9 and the conditions of these theorems, and show that we can not
omit the conditions of these theorems.

\begin{exam}
 Let $X=[0,1]$ with metric induced by $\Bbb{R}$: $d(x,y)=|x-y|$.
 Obviously $(X,d)$ be a complete metric space and the function $S:X \longrightarrow
 X$ by $Sx=\frac{x^{2}}{\sqrt{2}}$ is not contractive. If $T:X \longrightarrow X$ define by
 $Tx=x^{2}$ then $S$ is $T-contractive$, because:
 \[ \big{|} TSx-TSy \big{|}=\big{|} \frac{x^{4}}{2}-\frac{y^{4}}{2}\big{|} =\frac{1}{2} \big{|} x^{2}+y^{2} \big{|}
 \big{|} Tx-Ty \big{|} < \big{|} Tx-Ty \big{|}.\]
 So by Theorem 2.9 $S$ has a unique fixed point.
\end{exam}

\begin{exam}
 Let $X=[0,1]$ with metric induced by $\Bbb{R}$:
 $d(x,y)=|x-y|$. Obviously $(X,d)$ is a compact metric space. Let
 $T,S:X \longrightarrow X$ define by $Tx=x^{2}$ and
 $Sx=\frac{1}{2}\sqrt{1-x^{2}}$. Clearly $S$ is not contraction,
 but $S$ is $T-contraction$ and hence is $T-contractive$. Also $T$ is one-to-one. So by Theorem
 2.8 $S$ has a unique fixed point.
\end{exam}

\begin{exam}
 Let $X=[1,+\infty)$ with metric induced by $\Bbb{R}$:
 $d(x,y)=|x-y|$, thus, since $X$ is a closed subset of $\Bbb{R}$,
 it is a complete metric space. We define $T,S:X \longrightarrow
 X$ by $Tx=\ln x+1$ and $Sx=2\sqrt{x}$. Obviously, for every $n \in
 \Bbb{N}$, $S^{n}$ is not contraction. But we have,
 \[\big{|} TSx-TSy \big{|} =\frac{1}{2}\big{|} \ln x-\ln y \big{|}=\frac{1}{2}\big{|} Tx-Ty \big{|} \leq
 \frac{1}{2}\big{|} Tx-Ty \big{|}.\]
 Hence, $S$ is $T-contraction$.

 Also $T$ is one-to-one and subsequentially convergent. Therefore, by
 Theorem 2.5 $S$ has a unique fixed point.
\end{exam}

The following examples show that we can not omit the conditions of
Theorem 2.6 and Theorem 2.9.

In the following note we have two examples such that show that we
can not omit the one-to-one of $T$ in Theorem 2.6 and Theorem 2.9.
In first example $S$ has more than one fixed point and in the
second example $S$ has not a fixed point.\\

\textbf{Note 2.} Let $X=\{0,\frac{1}{2},1 \}$ with metric
$d(x,y)=|x-y|$. For functions $T_{1},S_{1}:X \longrightarrow X$
defined by $T_{1}x= \left\lbrace
  \begin{array}{c l}
     0 & \text{ $x=0,1$} \\
     \frac{1}{2} & \text{$x=\frac{1}{2}$}
  \end{array}
  \right.$ and $S_{1}x= \left\lbrace
  \begin{array}{c l}
     0 & \text{ $x=0,\frac{1}{2}$} \\
     1 & \text{$x=1$}
  \end{array}
  \right.$ we have $T_{1}$ is subsequentially convergent and since
  \[ \big{|} T_{1}S_{1}x-T_{1}S_{1}y \big{|} \leq \frac{1}{2} \big{|} T_{1}x-T_{1}y \big{|} \:\: (\forall x,y \in X),\]
  $S_{1}$ is $T_{1}-contraction$. But $T_{1}$ is not one-to-one and $S_{1}$ has two fixed
  points.\\
 If we define the functions $T_{2},S_{2}:X \longrightarrow X$ by
 $T_{2}x= \left\lbrace
  \begin{array}{c l}
     0 & \text{ $x=0,1$} \\
     \frac{1}{2} & \text{$x=\frac{1}{2}$}
  \end{array}
  \right.$
  and\\
  $S_{2}x= \left\lbrace
  \begin{array}{c l}
     1 & \text{ $x=0,\frac{1}{2}$} \\
     0 & \text{$x=1$}
  \end{array}
  \right.$
  then we have $T_{2}$ is subsequentially convergent and since
  \[ \big{|} T_{2}S_{2}x-T_{2}S_{2}y \big{|} \leq \frac{1}{2} \big{|} T_{2}x-T_{2}y \big{|} \:\: (\forall x,y \in X),\]
  $S_{2}$ is $T_{2}-contraction$. But $T_{2}$ is not one-to-one and $S_{2}$ has not a fixed
  point.\\

The following example shows that we can not omit the
subsequentially convergent of $T$ in Theorem 2.6.

\begin{exam}
 Let $X=[0,+\infty)$ with metric induced by $\Bbb{R}$: $d(x,y)=|x-y|$.
 Obviously $(X,d)$ be a complete metric space. For functions $T,S:X \longrightarrow
 X$ defined by $Sx=2x+1$ and $Tx=\exp(-x)$ we have, $T$ is one-to-one and
 $S$ is $T-contraction$ because:
\begin{align*}
 \big{|} TSx-TSy \big{|} & =\big{|} \exp(-2x-1)-\exp(-2y-1) \big{|} =\frac{1}{e}\big{|}
 \exp(-x)+\exp(-y)\big{|} \\ &
 \big{|} \exp(-x)-\exp(-y)\big{|} \leq \frac{2}{e}\big{|} \exp(-x)-\exp(-y) \big{|}=\frac{2}{e}\big{|} Tx-Ty \big{|}.
\end{align*}
 But $T$ is not subsequentially convergent $(Tn \underset{n\longrightarrow\infty}\longrightarrow 0$ but
 $\{n\}_{1}^{\infty}$ has not any convergence subsequence$)$ and $S$ has not a fixed
 point.
\end{exam}


Email:

A-Beiranvand@Arshad.araku.ac.ir

S-Moradi@araku.ac.ir

M-Omid@Arshad.araku.ac.ir

H-Pazandeh@Arshad.araku.ac.ir

\end{document}